\documentclass[leqno]{article}
\usepackage[english,frenchb]{babel}
\usepackage{lmodern}
\usepackage{enumerate}
\usepackage{amsmath}
\usepackage{amsfonts}
\usepackage{amssymb}
\usepackage{amsthm}
\usepackage{graphicx}
\usepackage{mathrsfs}
\usepackage{pxfonts}
\usepackage{datetime}
\usepackage{comment}
\usepackage{fourier-orns}
\usepackage{dsfont}
\usepackage[left=3cm,right=3cm,top=3cm,bottom=3cm]{geometry}
\usepackage[cyr]{aeguill}
\usepackage{fancyhdr}
\usepackage{mathabx}
\usepackage[all,cmtip]{xy}
\usepackage[latin1]{inputenc}

\usepackage{tikz-cd}
\usetikzlibrary{decorations.pathmorphing}

\usepackage{amsfonts,amsmath,graphicx}
\newcommand{\bc}{\begin{center}}
\newcommand{\ec}{\end{center}}
\newcommand{\bq}{\begin{quote}}
\newcommand{\eq}{\end{quote}}
\newcommand{\bqtn}{\begin{quotation}}
\newcommand{\eqtn}{\end{quotation}}
\newcommand{\beq}{\begin{equation}}
\newcommand{\eeq}{\end{equation}}
\newcommand{\bearr}{\begin{eqnarray}}
\newcommand{\eearr}{\end{eqnarray}}
\newcommand{\bearrn}{\begin{eqnarray*}}
\newcommand{\eearrn}{\end{eqnarray*}}
\newcommand{\bi}{\begin{itemize}}
\newcommand{\ei}{\end{itemize}}
\newcommand{\be}{\begin{enumerate}}
\newcommand{\ee}{\end{enumerate}}
\newcommand{\bthe}{\begin{theorem}}
\newcommand{\ethe}{\end{theorem}}
\newcommand{\blem}{\begin{lemme}}
\newcommand{\elem}{\end{lemme}}
\newcommand{\bsolu}{\begin{solution}}
\newcommand{\esolu}{\end{solution}}
\newcommand{\bexer}{\begin{exercise}}
\newcommand{\eexer}{\end{exercise}}

\newcommand{\ba}{\begin{array}}
\newcommand{\ea}{\end{array}}

\usepackage{amsfonts,amsmath}
\newtheorem{theoreme}{Theorem}[section]
\newtheorem{theorem}[theoreme]{Theorem}
\newtheorem{lemme}[theoreme]{Lemma}
\newtheorem{lemma}[theoreme]{Lemma}
\newtheorem{proposition}[theoreme]{Proposition}
\newtheorem{definition}[theoreme]{Definition}

\newtheorem{corollaire}[theoreme]{Corollary}

\newtheorem{solution}[theoreme]{Solution}
\newtheorem{exercise}[theoreme]{Exercise}
\newcommand{\bdefi}{\begin{definition}}
\newcommand{\edefi}{\end{definition}}
\newcommand{\brk}{\begin{remarque}}
\newcommand{\erk}{\end{remarque}}
\newcommand{\bpp}{\begin{proposition}}
\newcommand{\epp}{\end{proposition}}
\newcommand{\bpf}{\begin{proof}}
\newcommand{\epf}{\end{proof}}
\newcommand{\bcor}{\begin{corollaire}}
\newcommand{\ecor}{\end{corollaire}}
\newcommand{\bsol}{\begin{solution}}
\newcommand{\esol}{\end{solution}}
\theoremstyle{definition}

\newtheorem{remarque}[theoreme]{Remark}
\allowdisplaybreaks
\title{Generic polynomials for cyclic function field extensions over certain finite fields}
\author{Sophie Marques }

\begin{document}
\large
\selectlanguage{english}
\maketitle
\begin{abstract} 
In this paper, we find all the generic polynomials for geometric $\ell$-cyclic function field extensions over the finite fields $\mathbb{F}_q$ where $q= p^n$, $p$ prime integer such that $q \equiv -1 \mod \ell$ and $(\ell , p)=1$. 
\end{abstract}
Throughout the paper we will adopt the following notations unless mentioned differently.
$K$ denotes a one variable function field with field of constants $\mathbb{F}_q$ where $q= p^n$, $p$ prime integer. Given $x \in K$, we denote $\mathcal{O}_{K,x}$ is the integral closure of $\mathbb{F}_q[x]$ in $L$. Let $L/K$ be a cyclic Galois extension of degree $\ell$ with $(\ell , p)=1$. We write $\ell = \prod_{i=1}^s \ell_i^{f_i}$ where $\ell_i$ are distinct prime integers and $f_i$ positive integer. 
Let $\xi$ be a primitive $\ell^{th}$-root of unity. We suppose that $q\equiv -1 \mod \ell$. So that, by \cite[Theorem 2.10]{Con2}, $\mathbb{F}_q$ do not contains any primitive $\ell^{th}$ root of unity. More precisely, we have $[\mathbb{F}_q( \xi) : \mathbb{F}_q]=2$ and the minimal polynomial of $\xi$ over $ \mathbb{F}_q$ is $X^2 - (\xi + \xi^{-1} )X +1$ and $(\xi + \xi^{-1} )\in \mathbb{F}_q$. We denote by $\sigma$ the generator of $Gal(\mathbb{F}_q( \xi) ,\mathbb{F}_q)$, we have that $\sigma (\xi) = \xi^{-1}$. 

In this paper, our main result finds all the generic polynomials for geometric cyclic function field extensions over those finite fields $\mathbb{F}_q$ (Theorem \ref{m}). We find a one parameter family of generic polynomials of a cyclic extensions when $\ell$ is odd and a two parameters family of generic polynomials when $\ell$ is even (Corollary \ref{cm}). We also classify cyclic extensions up to isomorphism over those finite field $\mathbb{F}_q$ (Lemma \ref{clas}). Note that, in particular, this permits to classify all the geometric cyclic extensions of degree $3$, $4$ and $6$ over any finite fields $\mathbb{F}_q$. We describe the Galois action on a generator with a minimal polynomial in our form (Corollary \ref{gal}). We end the paper with the study of the ramification in term of our generation (Theorem \ref{ram}). In particular, we find that under our assumptions the ramified places are of even degree.

\section{Generic polynomials for cyclic extensions}
\begin{lemma} \label{1}
Let $L/K$ be a cyclic extension of degree $\ell$ with $q \equiv -1 \ mod \ \ell$.  Suppose $w$ is a Kummer generator for $L(\xi)/ K(\xi)$ whose minimal polynomial is of the form $X^{\ell} -a$, $y:=\sigma (w) + w$ is a generator for $L/K$ and $u:=\sigma(w) w \in K$. We write $\ell= 2 \iota + r$, where $r=0,1$. Then, the minimal polynomial of $y$ over $K$ is of the form
$$ P^{\ell}_{u , \alpha} (X)= X^\ell  -\ell u X^{\ell-2}  + \sum_{s=1}^{\iota} c_{s,\ell} u^s X^{\ell -2s}   - \alpha$$ 
where $\alpha = a + \sigma (a)$ and $ u^\ell = a \sigma(a)$, the coefficients $c^r_{s,j} \in \mathbb{F}_p$ and are defined recursively as
$$c_{s,j}^r= - \sum_{k=1}^{s} \binom{2j+r}{k} c_{s-k,j-1-k} \in \mathbb{F}_p$$
for $ 0 \leq s \leq j$ and $1 \leq j \leq \ell$ and $c_{0,j}^r=1$ for all $j \geq 0$. In particular, we have that $c_{1,j}^r= -(2j+r)$, $c_{j,j}^1= (-1)^{j}(2j+1)$ and  $c_{j,j}^0= (-1)^j2 $.\\
Defining, 
$$ Q^{l}_{u } (X)= X^l  -\ell u X^{l-2}  + \sum_{s=1}^{\iota} c_{s,l} u^s X^{l-2s} $$
for any $l$ integer where the $c_{s,l}$ are define recursively as before, we have that if $\ell= \prod_{i=1}^t l_i$ where $l_i$ are not necessarily distinct factors of $\ell$, then 
$$P^\ell_{u, \alpha} (X) = P_{l_t, u^{l/l_t} , \alpha} ( \cdots ( Q_{l_2, u^{l_1}  } ( Q_{l_1, u}(X))$$

\end{lemma} 
\begin{proof} 
Under the assumption of the theorem we determine the minimal polynomial of $y=w + \sigma (w)$. 
For if, $$\begin{array}{ccl} y^\ell &=& (w + \sigma (w))^\ell= \sum_{i=0}^\ell \binom{\ell}{i} w^i \sigma(w)^{\ell-i} \\
&=& w^\ell + \sigma(w)^\ell +\sum_{i=1}^{\ell-1} \binom{\ell}{i} w^i \sigma(w)^{\ell-i}\\
&=& a + \sigma (a ) +\sum_{i=1}^{\iota-(1-r)} \binom{\ell}{i} w^i \sigma(w)^{i}(\sigma(w)^{\ell-2i}+w^{\ell-2i}) + (1-r)\binom{\ell}{\iota} w^\iota \sigma (w)^\iota  \\ 
&=& a + \sigma (a) +\sum_{i=1}^{\iota-(1-r)} \binom{\ell}{i}u^i (\sigma(w)^{\ell-2i}+w^{\ell-2i}) + (1-r)\binom{\ell}{\iota} u^\iota \\ 
\end{array}$$ 
Note that $\ell - 2i$ is the same parity of $\ell$. 

We set $\omega_j^r= w^{2j+r}+ \sigma(w)^{2j+r} $ for $0 \leq j\leq \frac{\ell -1}{2}$ and $(j,r)\neq 0$ and $\omega_0^0=1$, then 
$$\omega_{j}^r=w^{2j+r}+ \sigma(w)^{2j+r} = y^{2j+r} - \sum_{k=1}^{j} \binom{2j+r}{k} u^k \omega_{j-k}^r $$
We have 
$$\omega_{0}^1 = y,$$
$$\omega_{1}^1 = y^3-3 u  y,$$
$$\omega_{2}^1 = y^5-5 u  y^3+5 u^2 y,$$
$$\omega_{1}^0 = y^2-2u$$
$$\omega_{2}^0 = y^4-4uy^2+2u^2$$
By induction in $j$ for $r=0,1$, we prove that 
$$\omega_j^r = \sum_{s=0}^j c_{s,j}^r u^s y^{2j+r-2s}$$
where the coefficients are obtained recursively as 
 $c_{s,j}^r=- \sum_{k=1}^{s} \binom{2j+r}{k} c_{s-k,j-k}^r $, $1\leq s \leq j$ and $j\geq 1$ and $c_{0,j}^r=1$ for all $j \geq 0$. We also prove that $c_{1,j}^r= -(2j+r)$, $c_{j,j}^1= (-1)^{j}(2j+1)$ and  $c_{j,j}^0= (-1)^j 2$.\\

From the computation above we see that this property is true for $j=1$. 
Suppose that this property is true for $\omega_t^r$, $1\leq t \leq j$, for some fixed $j$. We want to prove it remain true for $w_{j+1}^r$. Using the induction assumption, we obtain
$$\begin{array}{ccl} 
\omega_{j+1}&=& y^{2(j+1)+r} - \sum_{k=1}^{j+1} \binom{2(j+1)+r}{k}  u^k \omega_{j+1-k}^r \\
&=& y^{2(j+1)+r} -\sum_{k=1}^{j+1} u^k \binom{2(j+1)+r}{k}  \sum_{s=0}^{j+1-k}  c_{s,j+1-k}^r u^s y^{2(j+1)+r-2(s+k)}   \\
&=& y^{2(j+1)+r} - \sum_{k=1}^{j+1}   \sum_{s'=k}^{j+1} \binom{2(j+1)+r}{k}  u^{s'} c_{s'-k,j+1-k}^r y^{2(j+1)+r-2s'} \ \text{ where \ $s' := s+k$} \\
&=& y^{2(j+1)+r} +   \sum_{s'=1}^{j+1} (-\sum_{k=1}^{s'}   \binom{2(j+1)+r}{k} c_{s'-k,j+1-k}^r ) u^{s'} y^{2(j+1)+r-2s'}
\end{array}$$

Thus,  $c_{s,j+1}^r=- \sum_{k=1}^{s} \binom{2(j+1)+r}{k} c_{s-k,j+1-k}^r $, for $1\leq s \leq j+1$. 

Clearly, $c_{1,j+1}^r= -(2(j+1)+r)$. 

Note that 
\begin{equation}
 \begin{array}{ccl} 0 &=& (1-1)^{2n} \\ 
&=& \sum_{k=0}^{2n} \binom{2n}{k} (-1)^k \\
&=& \sum_{k=0}^n \binom{2n}{k} (-1)^k + \sum_{k=n+1}^{2n} \binom{2n}{k} (-1)^k \\ 
&=& \sum_{k=0}^n \binom{2n}{k} (-1)^k + \sum_{k=0}^{n-1} \binom{2n}{k} (-1)^k \\ 
&=& 2 \sum_{k=0}^{n-1} \binom{2n}{k} (-1)^k  + \binom{2n}{n} (-1)^n 
\end{array}
\end{equation}

Using the induction assumption, we have, 
$$\begin{array}{ccl} 
c_{j+1, j+1}^0 & =& - \sum_{k=1}^{j+1} \binom{2(j+1)}{k} c_{j+1-k,j+1-k}^0\\
&=& -2 (-1)^{j+1} \sum_{k=1}^{j} \binom{2(j+1)}{k} (-1)^k -  \binom{2(j+1)}{j+1}  \\
&=& 2 (-1)^{j+1}  -2 (-1)^{j+1} \sum_{k=0}^{j} \binom{2(j+1)}{k} (-1)^k -  \binom{2(j+1)}{j+1}  \\
&=&  (-1)^{j+1} 2  \text{ using $(1)$}
\end{array}$$
and 
$$\begin{array}{ccl} 
& & c_{j+1, j+1}^1\\
& =& - \sum_{k=1}^{j+1} \binom{2(j+1)+1}{k} c_{j+1-k,j+1-k}^0\\
&=&-  (-1)^{j+1} \sum_{k=1}^{j+1} \binom{2(j+1)+1}{k}  (-1)^{k}(2(j+1-k)+1)  \\
&=&  -(-1)^{j+1} \sum_{k=1}^{j+1} \binom{2(j+1)+1}{k}  (-1)^{k}(2(j+1)+1) + 2(-1)^{j+1} \sum_{k=1}^{j} \binom{2(j+1)+1}{k} k (-1)^{k}  \\
&=&  -(-1)^{j+1} (2(j+1)+1)( \sum_{k=1}^{j+1} (\binom{2(j+1)+1}{k}- 2\binom{2(j+1)}{k-1}) (-1)^{k}  ) \\
&=&  -(-1)^{j+1} (2(j+1)+1)( \sum_{k=1}^{j} \binom{2(j+1)}{k} (-1)^{k} - \sum_{k=1}^{j+1}  \binom{2(j+1)}{k-1}) (-1)^{k}) \\
&=&  -(-1)^{j+1} (2(j+1)+1) (\sum_{k=1}^{j+1} \binom{2(j+1)}{k} (-1)^{k} + \sum_{k=0}^{j}  \binom{2(j+1)}{k}) (-1)^{k}) \\
&=&  (-1)^{j+1} (2(j+1)+1) \text{ using (1)}
\end{array}$$
Moreover, note that when $\ell= \prod_{i=1}^t l_i$ where $l_i$ are non necessarily distinct factor of $\ell$, the minimal polynomial of $y$ over $L^{l_1}$ the fixed field of $L$ by the subgroup of $Gal(L/K)$ isomorphic to $\mathbb{Z}/ l_1 \mathbb{Z}$ is 
$$P^{l_1}_{u, \alpha_1}(X)$$ 
where $\alpha_1 = w^{l_1} + \sigma(w^{l_1} )$.

We let $Q^{l_1}_{ u}= P^{l_1}_{ u, \alpha_1}+ \alpha_1$. 
Since $y$ is a generator for $L/K$, we have that $[ K(w^{l_1} + \sigma(w^{l_1} )) : K] =\ell/ l_1$ and $w^{l_1}$ is a Kummer generator for $[ K(\xi )(w^{l_1} ) : K(\xi)]$, thus the minimal of $y$ over $L^{l_1l_2}$ the fixed field of $L$ by the subgroup of $Gal(L/K)$ isomorphic to $\mathbb{Z}/ l_1 l_2 \mathbb{Z}$ is 
$$P^{l_2}_{ u^{l_1} , \alpha_2 } ( Q_{l_1, u}(X))$$
where $\alpha_2 = w^{l_1l_2} + \sigma(w^{l_1l_2} )$.
Recursively, we obtain that the minimal polynomial $y$ of $L/K$ is 
$$P^{l_t}_{ u^{l/l_t} , \alpha} ( \cdots ( Q^{l_2}_{ u^{l_1}  } ( Q^{l_1}_{ u}(X))$$
where 
$Q^{l_k}_{ u}= P^{l_k}_{ u, \alpha_k}+ \alpha_k$ and $\alpha_k= w^{l_1l_2\cdots l_k} + \sigma(w^{l_1l_2\cdots l_k} )$.

By uniqueness of the minimal polynomial of $y$ over $K$, we obtain that
$$P^\ell_{u, \alpha} (X)= P_{l_t, u^{l/l_t} , \alpha} ( \cdots ( Q_{l_2, u^{l_1}  } ( Q_{l_1, u}(X))$$
\end{proof} 

\begin{remarque}
\begin{enumerate}
\item The previous lemma hold for general field extensions over a field of positive characteristic.
\item Note that when $2|\ell$, the polynomial $P^{\ell}_{u , \alpha}(X)$ only involve even degree monomials and when $(\ell , 2)=1$, the polynomial $P^{\ell}_{u , \alpha}(X)$ only involve odd degree monomials. 
\item We list some of those polynomials 
 \begin{itemize} 
\item[$\cdot$] $P^{3}_{u , \alpha}(X)=X^3-3uX- \alpha$, 
\item[$\cdot$]  $P^{5}_{u , \alpha}(X)=X^5-5u X^3+5u^2X - \alpha$,
\item[$\cdot$]  $P^{7}_{u , \alpha}(X)=X^7-7uX^5+14u^2 X^3-7u^3X-\alpha$, 
\item[$\cdot$]  $P^{9}_{u , \alpha}(X)= X^9-9uX^7+27u^2X^5-30u^3X^3+9u^4X-\alpha$, 
\item[$\cdot$]  $P^{11}_{u , \alpha}(X)=X^{11}-11uX^9+44u^2X^7-77u^3X^5+55u^4 X^3-11u^5 X- \alpha$ 
\item[$\cdot$]  $P^{13}_{u , \alpha}(X)=X^{13}-13uX^{11}+65u^2X^9-156u^3X^7+182u^4X^5-91u^5X^3+13 u^6X - \alpha$
\item[$\cdot$]  $P^{2}_{u , \alpha}(X)=X^2-2u-\alpha$, 
\item[$\cdot$]  $P^4_{u , \alpha}(X)= X^4-4uX^2+2u^2- \alpha$ 
\item[$\cdot$]  $P^6_{u , \alpha}(X)=X^6-6uX^4+9u^2X^2-2u^3 - \alpha$
\item[$\cdot$]   $P^8_{u , \alpha}(X)= X^8-8uX^6+20u^2X^4-16u^3X^2+2u^4- \alpha$.
\end{itemize}
We recognize $P^{3}_{u , \alpha}(X)$ when $u=1$ being the generic polynomials obtained in \cite{MWcubic} and \cite{MWcubic2}. We will prove that those polynomials $P^{\ell}_{u , \alpha}(X)$ are generic polynomials for cyclic extensions of degree $\ell$ over $\mathbb{F}_q$ when $q \equiv -1 \ mod \ 3$. 
\end{enumerate} 
\end{remarque}

\begin{lemma}\label{2} 
Let $L/K$ be a geometric cyclic extension of degree $\ell$ with $q \equiv -1 \ mod \ \ell$.  For any $z$  Kummer generator for $L(\xi)/ K(\xi)$, $z\sigma (z)\in K$. 
\end{lemma}
\begin{proof}
By \cite[Theorem 5.8.5]{Vil}, we know that $L(\xi ) / K(\xi)$ is a Kummer extension thus there exists a Kummer generator $z$ whose minimal polynomial is $X^\ell -a$ with $a \in K(\xi )$. 
Note that $\sigma (z)$ is also a Kummer generator for $L(\xi) / K(\xi)$ with minimal polynomial $X^\ell - \sigma(a)$.

We can write $L(\xi) / K(\xi)$ as a tower of prime degree extension 
$$z^{\ell_1} = w_{1, 1}, \cdots , w_{1,f_1}^{\ell_2}= w_{2,1} , \cdots , w_{s,f_s}^{\ell_s} = a$$ and we get a tower of cyclic prime degree extension 
$$L(\xi)/K(\xi)(w_{1,1}) /\cdots  \cdots / K(\xi)( w_{s,f_s})/ K(\xi ).$$ We set $w_{0,1}=z$, $w_{s+1, f_{s+1}}=a$ and  $f_0= f_{s+1}=1$. 

We want to prove that $z \sigma(z) \in K$. Since $\sigma (z \sigma (z) ) = z \sigma (z)$, $z\sigma ( z )\in L$ and we have 
$$(z\sigma(z))^{\ell_1} = w_{1, 1}\sigma(w_{1,1}), \cdots , (w_{1,f_1}\sigma (w_{1,f_1}))^{\ell_2}= w_{2,1} \sigma(w_{2,1}), \cdots , (w_{s,f_s}\sigma (w_{s,f_s}))^{\ell_s} = a\sigma (a)$$

Similarly, $w_{i,j} \sigma (w_{i,t_i}) \in K( w_{i,t_i})$, for any $1 \leq i \leq s$ and $1\leq t_i \leq f_i$. 

We have two cases, either all the $\ell_i$ are odd or one of the $\ell_i $ is equal to $2$, in which case we can suppose without loss of generality that $\ell_s=2$.
\begin{enumerate} 
\item {\sf Case 1: all $\ell_j$ are odd, for $1\leq j \leq s$.}

Since $(w_{i,j} \sigma (w_{i,j} )) ^{\ell_k} = w_{k, l}\sigma(w_{k,l})$, for any $0 \leq i \leq s$ and $1\leq t_i \leq f_i$ with 
\begin{itemize} 
\item[$\cdot$] $k= i+1$ and $l=1$, when $k=i$ and $j= f_i$, 
\item[$\cdot$] $l=j+1$ and $i=k$, if $ 0\leq j< f_i$, 
\end{itemize}
then either $w_{i,j} \sigma (w_{i,j}) \in K(w_{k, l}\sigma(w_{k,l}))$ or $K(w_{i,j} \sigma (w_{i,j} ))/ K(w_{k, l}\sigma(w_{k,l}))$ is a extension of degree $\ell_k$ in $L/K$. But the later case is impossible since $K(w_{i,j} \sigma (w_{i,j} ))/ K(w_{k, l}\sigma(w_{k,l}))$ is a subextension in the cyclic extension $L/K$  so it would be a Kummer extension, which is contradictory to the assumption that $q \equiv -1 \ mod\  \ell$ implying that $q \equiv -1 \ mod \ \ell_k$ (see \cite[Theorem 5.8.5]{Vil}). From this, we can prove that 
$$[K(z \sigma (z))/K ] = [ K (w_{1,1} \sigma (w_{1,1}) )/ K] = \cdots = [K(w_{s,f_s}\sigma (w_{s,f_s}))/K]=1$$ 
And thus, $z\sigma (z) \in K$. 
\item {\sf Case 2: one of the $\ell_i $ is equal to $2$, in which case we can suppose without loss of generality that $\ell_s=2$.}

From {\sf Case 1}, we know that $$[K(z \sigma (z))/K ] = [K(w_{s,1}\sigma (w_{s,1}))/K]$$
\begin{enumerate} 
\item If $f_s=1$, 
we have $(w_{s,1}\sigma (w_{s,1}))^2= a\sigma (a)$.

Thus either $w_{s,j}\sigma (w_{s,j}) \in K(w_{s,j+1}\sigma (w_{s,j+1}))$ or $w_{s,j}\sigma (w_{s,j} )$ is a Kummer generator for $K(w_{s,1}\sigma (w_{s,1}))/K$ and thus for $K(\xi)(w_{s,1}\sigma (w_{s,1}))/K(\xi)$. But then $w_{s,1}\sigma (w_{s,1})$ and $w_{s,1}$ are two Kummer generators for $K(\xi)(w_{s,1}\sigma (w_{s,1}))/K(\xi)$. Hence, by \cite[Theorem 5.8.5]{Vil}, there is $d\in K(\xi)$ such that $w_{s,1}\sigma (w_{s,1})= d w_{s,1}$, thus  $ w_{s,1}  = \sigma(d) \in K(\xi)$ which is impossible since $[K(\xi)(w_{s,1})/K(\xi)]=2$ by assumption. Thus $w_{s,1}\sigma (w_{s,1}) \in K$ and $[K(z \sigma (z))/K ]=1$ and $z\sigma(z)\in K$. 
\item If $f_s>1$, 
$$(w_{s,j}\sigma (w_{s,j}))^4 = w_{s,j+2}\sigma (w_{s,j+2})$$ 
for $1 \leq j \leq f_s-1$. 
Thus either $w_{s,j}\sigma (w_{s,j})$ is a Kummer generator of the degree $4$ subextension $K(w_{s,j}\sigma (w_{s,j}))/ K( w_{s,j+2}\sigma (w_{s,j+2}))$ of $L/K$ (since any subextension of a cyclic extension is cyclic) or $w_{s,j}\sigma (w_{s,j}) \in K(w_{s,j+1}\sigma (w_{s,j+1})$. 
But since $q \equiv -1 \ mod \ 2^{e_s}$ thus $q \equiv -1 \ mod \ 4$ and since $K$ does not contain a primitive $4^{th}$ root of unity $w_{s,j}\sigma (w_{s,j})$ cannot be a Kummer generator of degree $4$ and $w_{s,j}\sigma (w_{s,j}) \in K(w_{s,j+1}\sigma (w_{s,j+1})$. So that, we can prove that 
$$[K(z \sigma (z))/K ] = [K(w_{s,1}\sigma (w_{s,1}))/K]=[K(w_{s,2}\sigma (w_{s,2}))/K] = [K(w_{s,e_s}\sigma (w_{s,e_s}))/K]$$
with 
$$(w_{s,e_s} \sigma (w_{s,e_s}))^2 = a \sigma (a)$$ 
And using the same reasoning as in the case $(a)$ where $f_s=1$ we can deduce that $w_{s,e_s} \sigma (w_{s,e_s})\in K$ and finally, we have again $z \sigma (z)\in K$. 
\end{enumerate} 
\end{enumerate} 
\end{proof} 

\begin{lemma} \label{3} 
Let $L/K$ be a geometric cyclic extension of degree $\ell$ with $q \equiv -1 \ mod \ \ell$.  For any $z$ Kummer generator for $L(\xi)/ K(\xi)$, there is $\eta \in \mathbb{F}_q(\xi)^*$ such that  $\sigma (\eta z) + \eta z$ is a generator for $L/K$ and $\sigma (\eta ) \eta =1$. 
\end{lemma}
\begin{proof}
We write $z= z_1 + \xi z_2$, we have thus that $z_1, \ z_2 \in L$ and $z_2\neq 0$ since $z \notin L$ and not both $z_1$ and $z_2$ are in proper the subextension of $L/K$ over $K$ since $z$ is a generator for $L(\xi ) / K(\xi)$. 



We have $\xi + \xi^{-1} \neq \pm 2$, since if this equality was satisfied, we would have $X^2 - (\xi + \xi^{-1} )X +1= X^2 \mp 2 X +1$ is not irreducible over $\mathbb{F}_q$ contradicting that $q \equiv -1 \ mod \ \ell$ and thus $[\mathbb{F}_q(\xi): \mathbb{F}_q]=2$. 

Taking $\zeta = \frac{ \xi - \xi^{-1}  }{( \xi + \xi^{-1} )^2-4 },$
$$  \zeta z + \sigma (\zeta z)  = \frac{ \xi - \xi^{-1}  }{ ( \xi + \xi^{-1} )^2-4 } (z_1 +\xi z_2) + \frac{ \xi^{-1} - \xi  }{ ( \xi + \xi^{-1} )^2-4 } (z_1 +\xi^{-1} z_2)=  z_2. $$ 
Taking $\zeta' = \frac{ \xi^{-2} - 1  }{( \xi + \xi^{-1} )^2-4 }$, we get
$$  \zeta' z + \sigma (\zeta' z)  = \frac{ \xi^{-2} - 1 }{ ( \xi + \xi^{-1} )^2-4 } (z_1 +\xi z_2) + \frac{ \xi^{2} - 1 }{ ( \xi + \xi^{-1} )^2-4 } (z_1 +\xi^{-1} z_2)=  z_1. $$ 

For any $1 \leq k \leq f_i$ and any $1 \leq i \leq s$, we denote $L^{\ell_i^k}$ be the fixed field of $L$ by the unique subgroup of $Gal(L/K)$ isomorphic to $\mathbb{Z}/\ell_i^{k} \mathbb{Z}$. Since $z$ is a generator for $L(\xi)$ over $K(\xi)$, then $z$ is also a generator for $L(\xi)$ over $L^{\ell_i^k}(\xi ) $, for any $1 \leq k \leq f_i$ and any $1 \leq i \leq s$. 
Given $ 1 \leq i \leq s$, either $z_1$ or $z_2$ is a generator of $L/L^{\ell_i^{f_i}}$, indeed if that was not the case then both $z_1$ and $z_2$ would belong to a proper subextension of $L/L^{\ell_i^{f_i}}$ but since any such proper extension is contained in $L^{\ell_i^{f_i-1}}/K$. Then $z_1$ and $z_2$ would both belong to $L^{\ell_i^{f_i-1}}$ and $z=z_1 + \xi z_2 \in L^{\ell_i^{f_i-1}}(\xi)$ contradicting that $z$ is a generator for $L(\xi ) / L^{\ell_i^{f_i}}(\xi )$. 

That also implies in this case, that either $\mathfrak{z}_1=z + \sigma (z)$ or $\mathfrak{z}_2=\xi^{-1} z+ \sigma (\xi^{-1} z)$ is a generator for $L/L^{\ell_i^{f_i}}$. Indeed, $z + \sigma (z) = 2 z_1 + (\xi + \xi^{-1} ) z_2$ and $\xi^{-1} z + \sigma (\xi^{-1} z) = 2 z_2 + (\xi + \xi^{-1} ) z_1$. As before, if none of them was a  generator for $L/L^{\ell_i^{f_i}}$, we would have that they both belong to $L^{\ell_i^{f_i-1}}$ but then the same would be true for
$$\begin{array}{ccl} && (\xi + \xi^{-1} ) (\xi^{-1} z + \sigma (\xi^{-1} z) )- 2 (z + \sigma ( z) )\\ &=&  [(\xi + \xi^{-1} )^2-4] z_2\\ &=& (\xi + \xi^{-1} -2)  (\xi + \xi^{-1} +2) z_2 \end{array}$$ 
and for 
$$\begin{array}{ccl} && (\xi + \xi^{-1} )  (z + \sigma ( z) )-2(\xi^{-1} z + \sigma (\xi^{-1} z) ) \\ &=&  [(\xi + \xi^{-1} )^2-4] z_1\\ &=& (\xi + \xi^{-1} -2)  (\xi + \xi^{-1} +2) z_1 \end{array}$$ 
Contradicting that either $z_1$ or $z_2$ is a generator for $L/L^{\ell_i^{f_i}}$, since $\xi + \xi^{-1}\neq \pm 2$.\\

When $s=1$, the theorem is proven by the above.

Otherwise, when $s>1$, note that if we have $w\in L$ generator for $L$ over $L^{\ell_i^{f_i}}$ and $L^{\ell_j^{f_j}}$ for $i\neq j$. Then $[L: L^{\ell_i^{f_i}}]=[L^{\ell_i^{f_i}}(w) : L^{\ell_i^{f_i}}]= \ell_i^{f_i}$ and $[L: L^{\ell_j^{f_j}}]=[L^{\ell_j^{f_j}}(w) : L^{\ell_j^{f_j}}]= \ell_j^{f_j}$ and $ \ell_i^{f_i}  \ell_j^{f_j}| [L^{\ell_i^{f_i}\ell_j^{f_j}}(w) : L^{\ell_i^{f_i}\ell_j^{f_j}}]$ where $L^{\ell_i^{f_i}\ell_j^{f_j}}$ is the fixed field of $L$ by the unique subgroup of $Gal(L/K)$ isomorphic to $\mathbb{Z}/(\ell_i^{f_i} \ell_j^{f_j})\mathbb{Z}$ and $L=L^{\ell_i^{f_i}\ell_j^{f_j}}(w)$. Thus, $\mathfrak{z}_1$ is a generator for $L/L^I$ and $\mathfrak{z}_2$ is a generator for $L/L^J$ where $I$ and $J$ are subsets of $\{ 1 , \cdots , s\}$ such that $I\cup J=\{ 1 , \cdots , s\}$ with $I$ and $J$ maximal subsets having this property and defining $L^I$ to be the fixed field of $L$ by the unique subgroup of $Gal(L/K)$ isomorphic to $\mathbb{Z}/ (\prod_{i \in I}\ell_i^{f_i}) \mathbb{Z}$ and $L^J$ to be the fixed field of $L$ by the unique subgroup of $Gal(L/K)$ isomorphic to $\mathbb{Z}/ (\prod_{i \in J}\ell_i^{f_i} ) \mathbb{Z}$. \\

As a consequence, if $\mathfrak{z}_1 = k \mathfrak{z}_2$ for some $k \in K^*$ then $\mathfrak{z}_1=z + \sigma (z)$ and $\mathfrak{z}_2=\xi^{-1} z+ \sigma (\xi^{-1} z)$ are both generators for $L/K$ and the theorem is proven. \\

Also, if either $I$ or $J$ is the all $\{1,\cdots , s\}$ then either $\mathfrak{z}_1=z + \sigma (z)$ or $\mathfrak{z}_2=\xi^{-1} z+ \sigma (\xi^{-1} z)$ is a generator for $L$ over $K$ and the theorem is again proven.\\

Otherwise, there is $i_0 \in \{ 1, \cdots , s\} \backslash I$ and $i_1 \in \{ 1, \cdots , s\} \backslash J$ and $i_0\neq i_1$. 
Note also that since by assumption $\ell= \prod_{i=1}^s \ell_i^{f_i} | q-1$ then $q-1\geq \ell > s$. Indeed, $\ell$ is bigger or equal to the product of the $s$ smallest prime numbers and the $s^{th}$ smallest prime number is already bigger than $s$.  We denote $U_{q+1}$ the set of the $(q+1)^{th}$ roots of unity in $\mathbb{F}_q(\xi)^*$, we have $U_{q+1}= q+1>s+2$. 
Note that by assumption $\xi$ and $\xi^{-1}$ is a root of $R(X)=X^2 - (\xi + \xi^{-1}) X+1$ with $ (\xi + \xi^{-1}) \in \mathbb{F}_q$. Hence, we have also that $\xi^q$ is a root of $R(X)$ and thus either $\xi= \xi^q$ or $\xi^{-1} = \xi^q$ but since $\xi\notin \mathbb{F}_q$, we have $\xi^{-1} = \xi^q$. So that, for $\chi=c_1 + \xi c_2\in U_{q+1}$, we have 
$$\chi \sigma (\chi ) = (c_1+ \xi c_2) (c_1+ \xi^{-1} c_2)= (c_1+ \xi c_2) (c_1+ \xi^{q} c_2)= (c_1+ \xi c_2)^{q+1}=1$$ 

 Moreover, $\mathfrak{z}_1 , \ \mathfrak{z}_2$ and $\mathfrak{z}_1 + \chi \mathfrak{z}_2$ are distinct elements in $L$, for $\chi \in \mathbb{F}_q(\xi)^*$ and $\xi^{-1} \chi +1\in U_{q+1}$, since by assumption, $\mathfrak{z}_1 \neq k \mathfrak{z}_2$, for any $k, l \in K$. \\ 

Let $\chi_1 \in \mathbb{F}_q(\xi)^*$ such that $1 + \chi_1 \xi^{-1} \in U_{q+1}$ and $1 \leq i \leq s$, then either $\mathfrak{z}_1$ or $\mathfrak{z}_1 + \chi_1 \mathfrak{z}_2$ is a generator for $L/ L^{\ell_i^{f_i}}$, indeed, otherwise as before $\mathfrak{z}_1$ and $\mathfrak{z}_1 + \chi_1 \mathfrak{z}_2$ would both belong to $ L^{\ell_i^{f_i-1}}$, and $\chi_1 \mathfrak{z}_2 = \mathfrak{z}_1 + \chi_1 \mathfrak{z}_2- \mathfrak{z}_1\in L^{\ell_i^{f_i-1}}$, then $\mathfrak{z}_1$ and $\mathfrak{z}_2$ would both belong to $L/ L^{\ell_i^{f_i-1}}$ again contradicting that $z$ generates $L$ over $K$ as proven above. Thus, $\mathfrak{z}_1 + \chi_1 \mathfrak{z}_2$ is a generator of $L/L^{J_1}$ where $J_1$ is a subset of $\{ 1 , \cdots , s\}$ such that $I\cup J_1=\{ 1 , \cdots , s\}$ with $J_1$ maximal subset having this property and $L^{J_1}$ is the fixed field of $L$ by the unique subgroup of $Gal(L/K)$ isomorphic to $\mathbb{Z}/ (\prod_{i \in J_1}\ell_i^{f_i} )\mathbb{Z}$. Similarly, given $1\leq i \leq s$, we can prove that either $\mathfrak{z}_2$ or $\mathfrak{z}_1 + \chi_1 \mathfrak{z}_2$ is a generator for $L/ L^{\ell_i^{f_i}}$ and thus $J\cup J_1= \{1 ,\cdots , s\}$. \\

If $J_1$ is all $\{1 , \cdots , s\}$ then $\mathfrak{z}_1 + \chi_1 \mathfrak{z}_2 = \eta z + \sigma ( \eta z )$ is a generator $L/K$ where $\eta = 1 + \chi_1 \xi^{-1} \in U_{q+1} $ thus $\eta \sigma (\eta )=1$,  proving the theorem. \\

Otherwise, there is $i_2 \in \{1 ,\cdots , s\} \backslash J_1$. But since $J\cup J_1= \{1 ,\cdots , s\}$ and $I\cup J_1= \{1 ,\cdots , s\}$ then $i_0$ and $i_1\in J_1$ thus $i_0$, $i_1$ and $i_2$ are distinct and $|J_1|\geq 2$. 

We then let $\chi_2 \in \mathbb{F}_q(\xi)^*$ such that $1 + \chi_2 \xi^{-1} \in U_{q+1}$ and $\chi_1 \neq \chi_2$, as before we prove that $\mathfrak{z}_1 + \chi_2 \mathfrak{z}_2 $ is also a generator of $L/L^{J_2}$ where $J_2$ is a subset of $\{ 1 , \cdots , s\}$ such that $I\cup J_2=\{ 1 , \cdots , s\}$, $L^{J_2}$ is the fixed field of $L$ by the unique subgroup of $Gal(L/K)$ isomorphic to $\mathbb{Z}/ (\prod_{i \in J_2}\ell_i^{e_i} )\mathbb{Z}$ with $I\cup J_2= \{ 1, \cdots , s \}$ and $J \cup J_2= \{ 1, \cdots , s \}$ with $J_2$ maximal subset having this property. Given $1 \leq i \leq s$, we can also prove using the same argument as before that either $\mathfrak{z}_1 + \chi_1 \mathfrak{z}_2$ or $\mathfrak{z}_1 + \chi_2 \mathfrak{z}_2$ is a generator for $L/ L^{\ell_i^{e_i}}$ since $\mathfrak{z}_1 + \chi_1 \mathfrak{z}_2 - (\mathfrak{z}_1 + \chi_2 \mathfrak{z}_2) = (\chi_1 - \chi_2) \mathfrak{z}_2$ and $\frac{1}{\chi_1} (\mathfrak{z}_1 + \chi_1 \mathfrak{z}_2) - \frac{1}{\chi_2} (\mathfrak{z}_1 + \chi_2 \mathfrak{z}_2)= \frac{\chi_2 - \chi_1}{\chi_1\chi_2}\mathfrak{z}_1$. So that $J_1 \cup J_2 = \{ 1 , \cdots , s \}$. \\

If $J_2$ is all $\{1 , \cdots , s\}$ then $\mathfrak{z}_1 + \chi_2 \mathfrak{z}_2 = \eta z + \sigma ( \eta z )$ is a generator $L/K$ where $\eta = 1 + \chi_2 \xi^{-1} $ proving the theorem. \\

Otherwise, there is $i_3 \in \{ 1, \cdots , s\} \backslash J_1 $, and since $I\cup J_2 = J\cup J_2 = J_1 \cup J_2 =\{ 1, \cdots , s\}$ we have $i_0 ,i_1 , i_2 \in J_2$ and thus $i_0 , i_1 , i_2, i_3$  are all distinct and $|J_2| \geq 3$. Reproducing this process, with $\chi_3 , \cdots \chi_{s-1} \in  \mathbb{F}_q(\xi)^*$ such that $1 + \chi_i \xi^{-1} \in U_{q+1}$, for $1 \leq i \leq s-1$, all distinct and distinct from $\chi_1$ and $\chi_2$ which exist since $s< |U_{q+1}|$, we find that either $\mathfrak{z}_1 + \chi_k \mathfrak{z}_2$ is a generator for $L/K$ for some $k \in \{ 3, \cdots , s-2\}$ or $|J_{s-1}|\geq s$ and $\mathfrak{z}_1 + \chi_{s-1} \mathfrak{z}_2$ is a generator for $L/K$. In any case, we find $\eta \in U_{q+1}$ such that $\eta z + \sigma (\eta z)$ is a generator for $L/K$

\end{proof} 
\begin{theoreme} \label{m} 
Let $L/K$ be a geometric cyclic extension of degree $\ell$ with $q \equiv -1 \ mod \ \ell$. There exists $w$ a Kummer generator for $L(\xi)/ K(\xi)$ whose minimal polynomial is $X^\ell -a$ such that $y:=\sigma (w) + w$ is a generator for $L/K$ and $u:=\sigma(w) w \in K$ so that the minimal polynomial of $y$ over $K$ is $ P^{\ell}_{u , \alpha} (X)$ as in lemma \ref{1} where $ \alpha = a + \sigma (a)$. Conversely, if $L/K$ is a geometric extension of degree $\ell$ with $q \equiv -1 \ mod \ \ell$ and a generator $y$ whose minimal polymomial is of form $ P^{\ell}_{u , \alpha} (X)$ where $u , \alpha \in K$ such that $u^\ell= a \sigma  (a)$ and $\alpha = a+ \sigma (a)$, for some $a\in K(\xi)\backslash K$ then $L/K$ is a cyclic extension.
\end{theoreme}
\begin{proof} 
The first statement of the theorem is a direct consequence of Lemmas \ref{1}, \ref{2} and \ref{3}, noting that  if $z$ is a Kummer generator for $L(\xi ) /K(\xi)$ then $\eta z$ is also a Kummer generator for $L(\xi ) / K(\xi )$, by \cite[Theorem 5.8.5]{Vil}. \\

Conversely, suppose that $L/K$ is an extension of degree $\ell$ with $q \equiv -1 \ mod \ \ell$ and a generator $y$ whose minimal polymomial is of form $ P^{\ell}_{u , \alpha} (X)$ where $u , \alpha \in K$ such that $u^\ell= a \sigma  (a)$ and $\alpha = a+ \sigma (a)$, for some $a\in K(\xi)$.
Let $z_1, z_2 \in L^{alg}$ where $L^{alg}$ is the algebraic closure of $L$ the two roots of the polynomial $ X^2 - y X + u =0$.
Then $y = z_1 + z_2$ and $u = z_1 z_2$, by Lemma \ref{1}, we have that 
$$ z_1^\ell + z_2^\ell = P^{\ell}_{u, \alpha}(y) + \alpha = \alpha $$
And since $z_2 = \frac{u}{z_1}$, we have 
$$ z_1^\ell + \frac{u^\ell  }{z_1^\ell  }= \alpha $$
and 
$$z_1^{2\ell}- \alpha z_1^\ell + u^\ell$$ 
We have that 
$$\alpha = a + \sigma (a) \text{ and } u^\ell = a \sigma (a)$$ 
then $a$ and $\sigma (a)$ are root of the polynomial 
$$X^{2}- \alpha X+ u^\ell=0$$ 
Hence, up to reindexing, $z_1^\ell= a$ and $z_2^\ell = \sigma (a)$ and $\sigma (z_1) \neq z_1$ since $a \in K(\xi)\backslash K$.  Moreover, the coefficient of $T(X)=X^2-yX+u$ are in $L$, so that if $z_1$ is a root of $T(X)$ then $\sigma (z_1)$ is a root of $T(X)$ and thus, $z_2= \sigma (z_1)$. Since $\sigma (z_1) \neq z_1$, we also obtain that $\xi \in L(z_1)$ from $[L(z_1): L]\leq 2$ we can conclude that $ L(z_1)= L(\xi)$. 

Since $z_1^\ell= a\in K(\xi)$, $[ K(\xi) (z_1): K(\xi)]\leq \ell$. If $[ K(\xi) (z_1): K(\xi)]< \ell$, since $K(\xi)(z_1)= K(\xi) (\sigma (z_1))$, then $y:=z_1 + \sigma (z_1)$ would belong to a proper subextension of $L/K$, and $[K(y): K]<\ell$, which contradicts that $y$ is a generator for $L/K$. As a consequence, $z_1$ is a Kummer generator for $L(\xi)/K(\xi)$ and $L(\xi)/ K(\xi)$ is cyclic. The set of all the $\zeta z_1$ where $\zeta$ is a $\ell^{th}$ root of unity is the set of distinct roots for the polynomial $X^\ell - a$ and then $\zeta z_1 + \sigma( \zeta z_1)$ are the distinct roots of $P^\ell_{u, \alpha}(X)$. Indeed, $\zeta = \xi^i$ for some $0\leq i \leq \ell -1$, thus $\sigma (\zeta ) = \zeta^{-1}$, $ \zeta z_1 \sigma( \zeta z_1)= u$ and $ (\zeta z_1)^\ell + \sigma( \zeta z_1)^\ell = z_1^\ell + \sigma(z_1)^\ell= \alpha $. Moreover, if $\zeta_1 \neq \zeta_2 $ $\ell^{th}$ root of unity, such that $\zeta_1 z_1 + \sigma( \zeta_1 z_1)=\zeta_2 z_1 + \sigma( \zeta_2 z_1)$ then $\frac{ \zeta_1 - \zeta_2}{- \sigma (\zeta_1 -\zeta_2) }= \frac{\sigma (z_1)}{ z_1}= \frac{1}{u} \sigma (z_1)^2$, which contradicts that $L/K$ is a geometric extension. Proving that $\zeta z_1 + \sigma( \zeta z_1)$ are the distinct roots of $P^\ell_{u, \alpha}(X)$. This proves that $L/K$ is cyclic.

\end{proof}

\begin{lemma}\label{d}
 Let $\sigma$ a generator of $Gal( K(\xi ) /K)$ with $\xi$ a primitive $\ell^{th}$ root of unity. Let $d \in K (\xi )$, $\sigma (d) d =1$.
There is $A, B \in \mathcal{O}_{K,x}$ 
such that 
$$d = \frac{ A + \xi B}{A + \xi^{-1} B}= \frac{ C + \xi }{C + \xi^{-1} },$$ 
where $C= \frac{A}{B}$. 

When $K= \mathbb{F}_q(x)$, one can choose $(A, B)=1$. 
\end{lemma}
\begin{proof}
Let $\gamma \in K^*$. If we have $\gamma + d \sigma (\gamma) = 0$, then $ d= - \frac{\gamma}{\sigma (\gamma )}$. Noting that $\frac{ \xi - \xi^{-1}}{\sigma ( \xi - \xi^{-1})} =-1$. 
In this case, on can take $\theta = (\xi - \xi^{-1}) \gamma $ and obtain
$$d = \frac{ \theta}{\sigma (\theta )}.$$
Otherwise, let $\theta = \gamma + d \sigma (\gamma) \neq 0$, so that, 
\begin{align*} 
d\sigma (\theta )&= d \sigma (\gamma + d \sigma (\gamma)) \\
&=d \sigma(\gamma ) + d \sigma (d) \gamma \\ 
&= \gamma + d  \sigma(\gamma ) \\ 
&=\theta \end{align*}
We write $ \theta = \theta_1 + \xi \theta_2$, where $\theta_1, \ \theta_2 \in K$. One can write $\theta_1 = \frac{A_1}{B_1}$ and $\theta_2=  \frac{A_2}{B_2}$ where $A_1, B_1 , A_2 , B_2 \in \mathcal{O}_{K,x}$ so that $\theta = \frac{A_1}{B_1} + \xi \frac{A_2}{B_2} = \frac{A_1 B_2 + \xi  A_2 B_1}{B_1 B_2}$ and $\sigma ( \theta ) = \frac{A_1 B_2 + \xi^{-1}  A_2 B_1}{B_1 B_2}$, thus 
$$d = \frac{ \theta } {\sigma ( \theta )}= \frac{ A_1 B_2 + \xi A_2 B_1}{ A_1 B_2 + \xi^{-1}  A_2 B_1}$$ 
Taking $A= A_1 B_2$ and $B = A_2 B_1$, we obtain the theorem for general $K$. \\
When $K= \mathbb{F}_q(x)$, one can take $A= \frac{A_1 B_2}{gcd(A_1 B_2, A_2 B_1)}$ and $B= \frac{A_2 B_1}{gcd(A_1 B_2, A_2 B_1)}$ and prove the theorem, in this case. 
\end{proof}
\begin{corollaire} \label{sigm}
Let $L/K$ be a geometric cyclic extension of degree $\ell$ with $q \equiv -1 \ mod \ \ell$. There exists $w$ a Kummer generator for $L(\xi)/ K(\xi)$ whose minimal polynomial is $X^\ell -a$ such that $y:=\sigma (w) + w$ is a generator for $L/K$ and  so that the minimal polynomial of $y$ over $K$ is $ P^{\ell}_{u , \alpha} (X)$ as in lemma \ref{1} where $ \alpha = a + \sigma (a)$ and 
\begin{enumerate} 
\item when $\ell$ is an odd integer, $u= w \sigma (w) = a\sigma (a) =1$;
\item when $2 | \ell$, $\frac{a}{u^{\frac{\ell}{2}}} \sigma \big( \frac{a}{u^{\frac{\ell}{2}}} \big) =1$ where $u = w \sigma (w) \in K$ and 
$$P^{\ell}_{u , \alpha} (X)=P^{\frac{\ell}{2}}_{1 , \frac{ \alpha}{u^{\frac{\ell}{2}}}} \big( \frac{X^2-2u}{u}\big).$$ 
More precisely, $y$ generates the extension $L/L_2$ where $L_2$ is the fixed field of $L$ by the unique subgroup of $Gal(L/K)$ isomorphic to $\mathbb{Z}/ 2 \mathbb{Z}$ with generating equation:
$$y^2 =2 u + z^2 + \sigma (z)^2= u\big( 2 + \mathfrak{y} \big)$$
where  $\mathfrak{y}= \frac{\sigma (z)}{z}+\frac{z}{\sigma (z)}$ generates $L_2/K$ with minimal polynomial
$$ P^{\frac{\ell}{2}}_{1 , \frac{a}{u^{\frac{\ell}{2}}} } (X)$$
\end{enumerate} 
\end{corollaire}
\begin{proof}
\begin{enumerate} 
\item {\sf Suppose $\ell$ is an odd integer.}

Let $z$ be a Kummer generator for $L(\xi ) / K(\xi )$ whose minimal polynomial is $X^\ell -c$,  then $\frac{\sigma (z)}{z}$ is a Kummer generator $L(\xi ) / K(\xi )$. Indeed, if it was not then $v=\frac{\sigma (z)}{z}$ would belong to a proper subextension $L'$ of $L$ but since $z \sigma (z) \in K$ by Lemma \ref{2}, then $\frac{\sigma (z)}{z}z \sigma (z) = \sigma (z)^2\in L'$ but since $z$ is a Kummer generator for $L(\xi ) / K(\xi )$, so is $\sigma (z)$ and since $\ell$ is odd, $\sigma (z)^2$ is a Kummer generator of $L(\xi ) / K(\xi )$, by \cite[Proposition 5.8.7]{Vil} which leads to a contradiction. By Lemma \ref{3}, there exist $\eta \in \mathbb{F}_q(\xi)^*$ such that $\eta \sigma (\eta )=1$, 
$\eta v + \sigma ( \eta v)$ is a generator for $L/K$. Thus taking $w= \eta v$, we have $y:= w+ \sigma (w)$ is a generator for $L/K$ and $ w \sigma (w)=1$, $w^\ell= \eta^\ell \frac{\sigma (c)}{c}=:a$ and by Lemma \ref{1}, the minimal polynomial of $y$ is $P^{\ell}_{1 , \alpha}$ where $\alpha = a + \sigma (a )$.\\ 



\item {\sf Suppose $\ell$ is an even integer.} Thanks to Theorem \ref{m}, we can choose $w$ to be a Kummer generator for $L(\xi ) / K(\xi )$ whose minimal polynomial is $X^\ell -a$ such that $y := \sigma (w) + w$  is a generator for $L/K$ with minimal polynomial $P^\ell_{u, \alpha} (X)$ with $u = w \sigma (w)\in K$ and $\alpha = a+ \sigma (a)$.
Note that 
 $$ \frac{\sigma ( \omega^\ell)}{ \omega^\ell }  =\frac{ \sigma (a)^2}{u^\ell}$$ 
 thus 
  $$ \frac{\sigma (  \omega^{\frac{\ell}{2}})}{ \omega^{\frac{\ell}{2}} }  =\frac{ \sigma (a)}{u^{\frac{\ell}{2}}}$$ 
  And 
   $$\frac{ a}{u^{\frac{\ell}{2}}}= \frac{ \omega^{\frac{\ell}{2}} }{\sigma (  \omega^{\frac{\ell}{2}})}=  \frac{a^{\frac{\ell}{2}}}{  \sigma (a)}  $$ 
  $$\sigma \big(  \frac{ a}{u^{\frac{\ell}{2}}} \big)= \frac{u^{\frac{\ell}{2}}}{ a}.$$
  As a consequence, 
  $$ \frac{a}{u^{\frac{\ell}{2}}} \sigma \big( \frac{a}{u^{\frac{\ell}{2}}} \big) =1$$
Note that since $P^\ell_{u, \alpha} (X)$ only involves even monomial we can write 
$P^\ell_{u, \alpha} (X)= Q(X^2)$. 
where $Q(X)$ is a monic irreducible polynomial over $K$ and $deg (Q(X))=\frac{\ell}{2}$. Thus $Q(X)$ is the minimal polynomial of $y^2$,
$[K(y^2): K]= \frac{\ell}{2}$ and $L_2 = K(y^2)$ and $\mathfrak{y}:=\frac{y^2-2u}{u} = \frac{ \sigma (w)}{w} + \frac{w}{\sigma (w)}$ is also a generator for $L_2 /K$ since $u \in K$. Moreover, $\frac{\sigma (w)}{w} $ Kummer generator for $L_2/K$ such that $\big( \frac{\sigma (w)}{w} \big)^{\frac{\ell}{2}} =\frac{a}{u^{\frac{\ell}{2}}}$, thus by Lemma \ref{1}, the minimal polynomial of $\mathfrak{y}$ is $P^{\frac{\ell}{2}}_{1, \frac{a}{u^{\frac{\ell}{2}}}}(X)$ and the the result.
\end{enumerate}

\end{proof}

Combining together Theorem \ref{m}, Corollary \ref{sigm} and Lemma \ref{d}, we obtain:
\begin{corollaire}\label{cm}
 Let $L/K$ be a cyclic extension of degree $\ell$ with $q \equiv -1 \ mod \ \ell$. There exists $w$ a Kummer generator for $L(\xi)/ K(\xi)$ whose minimal polynomial is $X^\ell -a$ such that $y:=\sigma (w) + w$ is a generator for $L/K$ and $u :=\sigma(w) w \in K$ so that the minimal polynomial of $y$ over $K$ is
\begin{enumerate} 
\item $$ P^{\ell}_{1 , \alpha} (X)= X^\ell  -\ell  X^{\ell-2}  + \sum_{s=1}^{q} c_{s,\ell}  X^{\ell -2s}  - \frac{2 A^2 + 2 (\xi + \xi^{-1}) AB + (\xi^2 + \xi^{-2}) B^2 }{A^2 +(\xi + \xi^{-1} ) AB + B^2 }$$ for some $A, B  \in \mathcal{O}_{K,x}$ such that $a = \frac{\sigma (A+ \xi B)}{A+\xi B}$ and $u=1$, when $\ell$ is an odd integer;
\item $$ P^{\ell}_{u , \alpha} (X)= X^\ell  -\ell u X^{\ell-2}  + \sum_{s=1}^{q} c_{s,\ell} u^s X^{\ell -2s}  - u^{\frac{\ell }{2}}\frac{2 A^2 + 2 (\xi + \xi^{-1}) AB + (\xi^2 + \xi^{-2}) B^2 }{A^2 +(\xi + \xi^{-1} ) AB + B^2 }, $$  for some $A, B  \in \mathcal{O}_{K,x}$ such that $a =u^{\frac{\ell}{2}}  \frac{\sigma (A+ \xi B)}{A+\xi B}$, when $\ell$ is an even integer;
 \end{enumerate} 
 where $c_{s, \ell}$ are defined in Lemma \ref{1}.
\end{corollaire}
\begin{remarque} 
Note that 
$$\frac{2 A^2 + 2 (\xi + \xi^{-1}) AB + (\xi^2 + \xi^{-2}) B^2 }{A^2 +(\xi + \xi^{-1} ) AB + B^2 }= 
\frac{2 C^2 + 2 (\xi + \xi^{-1}) C + (\xi^2 + \xi^{-2})  }{C^2 +(\xi + \xi^{-1} ) C + 1}$$
where $C = \frac{A}{B}$. 
\end{remarque}
One can obtain the following result combining Theorem \ref{m} and Proposition 5.8.7. VS
\begin{lemma}\label{clas}
Let $L_1/K$ and $L_2/K$ be two geometric cyclic extensions of degree $\ell$ with $q \equiv -1 \ mod \ \ell$. For $i=1,2$, there exists $w_i$ a Kummer generator for $L(\xi)/ K(\xi)$ whose minimal polynomial is $X^\ell -a_i$ such that $y_i:=\sigma (w_i) + w_i$ is a generator for $L/K$, so that the minimal polynomial of $y_i$ over $K$ is $ P^{\ell}_{u_i , \alpha_i} (X)$ as in lemma \ref{1} where $ \alpha_i = a_i + \sigma (a_i)$ and $u_i:=\sigma(w_i) w_i \in K$. Then, the following assertions are equivalent: 
\begin{enumerate}
\item $L_1= L_2$
\item $L_1 (\xi ) = L_2 (\xi)$ 
\item $w_2 = c w_1^j $ so that $y_2 = c w_1^j  + \sigma (c)  \sigma (w_1) ^j $ for all $1 \leq j \leq n - 1$ such that $( j, n) = 1$ and $c \in K(\xi)$. 
\item $a_2 = c^\ell a_1^j $ so that $\alpha_2= c^\ell a_1^j + \sigma (c^\ell)  \sigma (a_1) ^j $, for all $1 \leq j \leq n - 1$ such that $( j, n) = 1$ and $c \in K(\xi)$. 
\end{enumerate} 
\end{lemma}
The following result is a direct corollary of the previous lemma. 
\begin{corollaire} \label{gal}
Let $L/K$ be a geometric cyclic extension of degree $\ell$ with $q \equiv -1 \ mod \ \ell$. There exists $w$ a Kummer generator for $L(\xi)/ K(\xi)$ whose minimal polynomial is $X^\ell -a$ such that $y:=\sigma (w) + w$ is a generator for $L/K$, so that the minimal polynomial of $y$ over $K$ is $ P_{\ell,u , \alpha} (X)$ as in lemma \ref{1} where $ \alpha = a + \sigma (a)$ and $u:=\sigma(w) w \in K$. Then the Galois group can be identified with the group of the $\ell^{th}$ root of unity, and the Galois action is given by $y \mapsto y_\zeta = \sigma (\zeta w) +\zeta w$
\end{corollaire}
\section{Ramification} 
\begin{theoreme} \label{ram}
Let $L/K$ be a geometric cyclic extension of degree $\ell$ with $q \equiv -1 \ mod \ \ell$. As in Theorem \ref{m}, we choose $w$ a Kummer generator for $L(\xi)/ K(\xi)$ whose minimal polynomial is $X^\ell -a$ such that $y:=\sigma (w) + w$ is a generator for $L/K$, so that the minimal polynomial of $y$ over $K$ is $ P_{\ell,u , \alpha} (X)$ as in lemma \ref{1} where $ \alpha = a + \sigma (a)$ and $u:=\sigma(w) w \in K$. Let $\mathfrak{p}$ be a place of $K$ and $\mathfrak{P}$ a place of $L$ above $\mathfrak{p}$. We denote $e(\mathfrak{P} | \mathfrak{p})$ the index of ramification of $\mathfrak{P}| \mathfrak{p}$,
We have
\begin{enumerate} 
\item[$\cdot$] when $v_\mathfrak{p} ( \frac{u^\ell }{\alpha^2 } )\geq0$ then
\begin{itemize} 
\item[$\cdot$] $v_{\mathfrak{p}} \big( \alpha \big) \neq 0$, $\mathfrak{p}$ is of even degree and $e(\mathfrak{P}| \mathfrak{p}) = \frac{ \ell}{(\ell  ,   v_{\mathfrak{p}} (\alpha ) )}.$
\item[$\cdot$] $v_{\mathfrak{p}} \big( \alpha \big) = 0$, $e(\mathfrak{P}| \mathfrak{p}) = 1$. 
\end{itemize}
\item[$\cdot$] when $v_\mathfrak{p} ( \frac{u^\ell }{\alpha^2 } )<0$ then
\begin{itemize} 
\item[$\cdot$]  $e(\mathfrak{P}| \mathfrak{p}) = 2$, if $(v_\mathfrak{p} ( u),2)=1$.
\item[$\cdot$] $e(\mathfrak{P}| \mathfrak{p}) = 1$, otherwise.
\end{itemize}
\end{enumerate} 
When $\ell$ is odd, one can choose $y$ such that $ u=1$, $\mathfrak{p}$ is unramified if and only if $v_{\mathfrak{p}} (\alpha ) \geq 0$. Moreover, when $\mathfrak{p}$ is ramified ,$\mathfrak{p}$ is of even degree and
 $$e(\mathfrak{P}| \mathfrak{p}) = \frac{ \ell}{(\ell  ,   v_{\mathfrak{p}} (\alpha ) )}.$$ 
\end{theoreme}
\begin{proof}
Let $\mathfrak{p}$ be a place of $K$, $\mathfrak{p}_\xi$ a place of $K(\xi)$ above $\mathfrak{p}$ and $\mathfrak{P}$ a place of $L$ above $\mathfrak{p}$
From \cite[Theorem 5.8.12]{Vil}, since $K(\xi )/K$ is unramified and the index of ramification is multiplicative in tower, $\mathfrak{p}$ is unramified in $L$ if and only if $\mathfrak{p}_\xi$ is unramified in $L(\xi)$. That is, $\ell | v_{\mathfrak{p}_\xi} (a)$. Thus, in order to prove the theorem, we need to determine $v_{\mathfrak{p}_\xi} (a)$ in terms of $v_{\mathfrak{p}}(\alpha )$ and $v_{\mathfrak{p}}(u )$.  
For if, note that $a$ satisfies the equation 
$$a^2 - \alpha a + u^\ell =0$$ 
since $\alpha = a + \sigma (a)$ and $u^\ell = a \sigma (a)$.\\ 
Moreover $K(a) = K(\xi)$ since $a \in K(\xi)\backslash K$, thus $K(a) / K$ is unramified. \\ 
We have 
$$\big( \frac{a}{\alpha}\big) ^2 - \frac{ a}{\alpha} = - \frac{u^\ell }{\alpha^2 } $$ 

When $v_\mathfrak{p} ( \frac{u^\ell }{\alpha^2 } )<0$, we find by the triangular inequality for valuation that $v_{\mathfrak{p}_\xi} \big( \frac{a}{\alpha}\big) <0$ and $2v_{\mathfrak{p}_\xi} \big( \frac{a}{\alpha}\big) = v_\mathfrak{p} ( \frac{u^\ell }{\alpha^2 } )$ then 
$$v_{\mathfrak{p}_\xi} \big(a) = v_{\sigma (\mathfrak{p}_\xi )} \big(a) =\frac{\ell v_\mathfrak{p} ( u )}{2}.$$

When $v_\mathfrak{p} ( \frac{u^\ell }{\alpha^2 } )=0$,  by the triangular inequality, we obtain that 
$v_{\mathfrak{p}_\xi} \big( \frac{a}{\alpha}\big) =0$, thus 
$$v_{\mathfrak{p}_\xi} \big(a) = v_{\mathfrak{p}_\xi} \big(\alpha \big) =  v_{\mathfrak{p}} \big(\alpha \big) $$

When $v_\mathfrak{p} ( \frac{u^\ell }{\alpha^2 } )>0$,  then by Kummer theorem applied to the polynomial $X^2 - X + \frac{u^\ell}{ \alpha^2 }$, we know that $\mathfrak{p}$ splits in $K(\xi )$, therefore, by \cite[Theorem 6.2.1]{Vil}, we obtain that $ \mathfrak{p} $ is of even degree. Moreover, by the triangular inequality $v_{\mathfrak{p}_\xi}  \big( \frac{a}{\alpha}\big) \geq 0$, for any place $\mathfrak{p}_\xi$ of $K(\xi )$ over $\mathfrak{p}$. More precisely, either $v_{\mathfrak{p}_\xi}  \big( \frac{a}{\alpha}\big) = 0$ or   $v_{\mathfrak{p}_\xi}  \big( \frac{a}{\alpha}\big)= v_{\mathfrak{p}} \big( \frac{u^\ell }{\alpha^2} \big)$. Moreover,
$$\frac{u^\ell }{\alpha^2} =   \frac{a}{\alpha} \sigma  \big( \frac{a}{\alpha}\big)$$
Thus, 
$$v_{\mathfrak{p}} \big( \frac{u^\ell }{\alpha^2} \big)= v_{\mathfrak{p}_\xi} \big( \frac{a}{\alpha} \big) + v_{\mathfrak{p}_\xi}  \big( \sigma \big( \frac{a}{\alpha} \big) \big) = v_{\mathfrak{p}_\xi} \big( \frac{a}{\alpha} \big) + v_{\sigma (\mathfrak{p}_\xi)}   \big( \frac{a}{\alpha} \big) $$
From this we deduce that either 
$$v_{\mathfrak{p}_\xi} \big( \frac{a}{\alpha} \big) = 0 \text{ and } v_{\sigma (\mathfrak{p}_\xi)}   \big( \frac{a}{\alpha} \big) =v_{\mathfrak{p}} \big( \frac{u}{\alpha^2} \big)
$$
that is,
$$v_{\mathfrak{p}_\xi} \big( a \big) = v_{\mathfrak{p}} \big( \alpha \big)  \text{ and } v_{\sigma (\mathfrak{p}_\xi)}   \big(a  \big) =\ell v_{\mathfrak{p}} \big(u \big)- v_{\mathfrak{p}} \big( \alpha \big)
$$
or 
$$v_{\mathfrak{p}_\xi} \big( \frac{a}{\alpha} \big) = v_{\mathfrak{p}} \big( \frac{u^\ell }{\alpha^2} \big) \text{ and } v_{\sigma (\mathfrak{p}_\xi)}   \big( \frac{a}{\alpha} \big) =0
$$
that is,
$$  v_{\mathfrak{p}_\xi}   \big(a  \big) =\ell v_{\mathfrak{p}} \big(u \big)- v_{\mathfrak{p}} \big( \alpha \big) \text{ and } v_{\sigma (\mathfrak{p}_\xi)} \big( a \big) = v_{\mathfrak{p}} \big( \alpha \big) 
$$
This permitting to obtain the desire result.
\end{proof} 

\begin{remarque}
When $\ell$ is odd, one can choose a single place $\mathfrak{P}_\infty$ at infinity in $K$ such that $v_{\mathfrak{P}_\infty}(a) \geq 0$ and thus $\mathfrak{P}_\infty$ is unramified. (see \cite[Corollary 3.12]{MWcubic2}) 
\end{remarque}
\begin{lemma} 
Let $\ell$ be an odd integer. Let $L/K$ be a cyclic extension of degree $\ell$ with $q \equiv -1 \ mod \ \ell$. Given $\mathfrak{p}$ a place of $K$. There is $w$ a Kummer generator for $L(\xi)/ K(\xi)$ whose minimal polynomial is $X^\ell -a$ such that $y:=\sigma (w) + w$ is a generator for $L/K$, so that the minimal polynomial of $y$ over $K$ is $ P^{\ell}_{1 , \alpha} (X)$ as in lemma \ref{1} where $ \alpha = a + \sigma (a)$, $\sigma(w) w =1$ and either $v_{\mathfrak{p}} (\alpha ) = - m$ where $0\leq m \leq \ell -1$ or $v_{\mathfrak{p}} (\alpha ) \geq 0$. 
 Moreover, when $v_{\mathfrak{p}} (\alpha ) =-m$ where $0\leq m \leq \ell -1$  then $\mathfrak{p}$ is ramified and $v_{\mathfrak{p}} (\alpha ) \geq 0 $ then $\mathfrak{p}$ is unramified.
\end{lemma} 
\begin{proof} 
Let $\mathfrak{p}$ a place of $K$. By Corollary \ref{cm}, there is $v$ a Kummer generator for $L(\xi)/ K(\xi)$ whose minimal polynomial is $X^\ell -c$ such that $z:=\sigma (v) + v$ is a generator for $L/K$ and $\sigma( v) v=1 $ so that the minimal polynomial of $z$ over $K$ is $ P^{\ell}_{1 , \beta} (X)$ as in lemma \ref{1} where $ \beta = c + \sigma (c)$. From the proof of the previous lemma \ref{ram}, we have that either $v_\mathfrak{p} (\beta ) \geq 0$ when $v_{\mathfrak{p}_\xi} (c )=0$ or $\mathfrak{p}$ split $K(\xi )$ and for one of the place $\mathfrak{p}_\xi$ above $\mathfrak{p}$, we have
$$v_\mathfrak{p} (\alpha ) = v_{\mathfrak{p}_\xi} (c )= -v_{\sigma (\mathfrak{p}_\xi) } (c ) <0$$
By Lemma \ref{d}, there is $\gamma \in K(\xi)^*$ such that 
$$ a = \frac{\sigma (\gamma )}{\gamma}. $$
We write $v_{\mathfrak{p}_\xi} (\gamma) = \ell s +m$ where $0 \leq m \leq \ell -1$, using weak approximation theorem we choose $\lambda \in K(\xi)$ such that $v_{\mathfrak{p}_\xi} (\lambda ) = -s$. So that $\frac{\sigma (\lambda )}{\lambda} v$ is a Kummer generator for $L(\xi)/ K(\xi)$ whose minimal polynomial is $X^\ell -\lambda^\ell c$ and
$$v_{\mathfrak{p}_\xi} (\big( \frac{\sigma (\lambda )}{\lambda} \big)^\ell c )=-v_{\sigma (\mathfrak{p}_\xi)} (\big( \frac{\sigma (\lambda )}{\lambda} \big)^\ell c ) =- m.$$
By Lemma \ref{3}, there is $\eta\in \mathbb{F}_q (\xi)^*$ such that $\sigma (\eta ) \eta =1$, $w= \eta \frac{\sigma (\lambda )}{\lambda} v $ is a Kummer generator $L(\xi ) /K(\xi)$ with minimal polynomial $X^\ell - a$ where $a = \eta^\ell \big( \frac{\sigma (\lambda )}{\lambda} \big)^\ell c$ and $y:=w+ \sigma (w )$ is a generator for $L/K$ and its minimal polynomial is $P^\ell_{1, \alpha } (X)$ with $\alpha= a + \sigma (a)$ and $v_{\mathfrak{p}} ( \alpha ) = -m $.

\end{proof}
\begin{lemma}\label{form}
Let $\ell$ be an odd integer. Suppose $K= \mathbb{F}_q(x)$. Let $L/K$ be a cyclic extension of degree $\ell$ with $q \equiv -1 \ mod \ \ell$. There is $w$ a Kummer generator for $L(\xi)/ K(\xi)$ whose minimal polynomial is $X^\ell -a$ such that $y:=\sigma (w) + w$ is a generator for $L/K$ so that the minimal polynomial of $y$ over $K$ is $ P^{\ell}_{1 , \alpha} (X)$ as in lemma \ref{1} where $ \alpha = a + \sigma (a)$ and $\sigma(w) w =1$ such that at any $\mathfrak{p}$ place of $K$ either $v_{\mathfrak{p}} (\alpha ) = - m$ where $0\leq m \leq \ell -1$ or $v_{\mathfrak{p}} (\alpha ) \geq 0$. 
 Moreover, when $v_{\mathfrak{p}} (\alpha ) =-m$ where $0\leq m \leq \ell -1$  then $\mathfrak{p}$ is ramified, $v_{\mathfrak{p}} (\alpha ) \geq 0 $ then $\mathfrak{p}$ is unramified and  $v_{\mathfrak{p}_\infty } (\beta )\geq 0$ where $\mathfrak{p}_\infty $ is the pole divisor of $x$ so that $\mathfrak{p}_\infty $ is unramified.  
\end{lemma} 
\begin{proof} 
By Corollary \ref{cm}, there is $w$ a Kummer generator for $L(\xi)/ K(\xi)$ whose minimal polynomial is $X^\ell -c$ such that $z:=\sigma (v) + v$ is a generator for $L/K$ so that the minimal polynomial of $z$ over $K$ is $ P^{\ell}_{1 , \beta} (X)$ as in lemma \ref{1} where $ \beta = c + \sigma (c)$ and $\sigma( v) v=1 $. 
By Lemma \ref{d}, there is $\gamma \in F_q(\xi)[x]^*$ such that $(\sigma (\gamma ) , \gamma)=1$ and 
$$ c = \frac{\sigma (\gamma )}{\gamma}. $$
Since $F_q(\xi)[x]$ is a unique factorization domain and we write 
$$\gamma = \prod_{i=1}^t \gamma_i^{e_i}$$
where $\gamma_i$ are distincts irreducible polynomials in $\mathbb{F}_q[x]$ and $e_i= \ell s_i +m_i$ where $0 \leq m_i \leq \ell -1$, for $0\leq i \leq t$. 
Since $(\sigma (\gamma ) , \gamma)=1$ we have $\sigma (\gamma_i ) \neq \gamma_i$. 

Let $\lambda =  \prod_{i=1}^t \gamma_i^{-s_i}$ so that $\theta = \lambda^\ell \gamma = \prod_{i=1}^t \gamma_i^{m_i}$ and 
$$a:=\frac{\sigma (\theta )}{\theta }=\frac{\sigma (\lambda^\ell \gamma )}{\lambda^\ell\gamma}= \frac{\prod_{i=1}^t \sigma(\gamma_i )^{m_i}}{\gamma_i^{m_i}}$$

Moreover, $ \frac{\sigma (\lambda )}{\lambda} v$ is a Kummer generator for $L (\xi) 
\ K(\xi)$. By Lemma \ref{3}, there is $\eta\in \mathbb{F}_q (\xi)^*$ such that $\sigma (\eta ) \eta =1$, $w= \eta \frac{\sigma (\lambda )}{\lambda} v $ is a Kummer generator $L(\xi ) /K(\xi)$ with minimal polynomial $X^\ell - a$ where $a = \eta^\ell \big( \frac{\sigma (\lambda )}{\lambda} \big)^\ell c$ and $y:=w+ \sigma (w )$ is a generator for $L/K$ and its minimal polynomial is $P^\ell_{1, \alpha } (X)$ with $$\alpha= a + \sigma (a) = \frac{ \gamma^2 + \sigma (\gamma )^2}{\gamma \sigma (\gamma)}=\frac{ \theta^2 + \sigma (\theta )^2}{\prod_{i=1}^t (\gamma_i \sigma (\gamma_i))^{m_i}}$$
 $\gamma_i \sigma (\gamma_i)$ is then an irreducible polynomial in $\mathbb{F}_q[x]$ and $ (\theta \sigma (\theta ), \theta^2 + \sigma (\theta )^2)=1$ so that if $\mathfrak{p}_i$ is the finite place of $K$ corresponding to $\gamma_i \sigma (\gamma_i)$, $v_{\mathfrak{p}_i} (a + \sigma (a) )= - m_i$. Clearly, at any other finite place, $ v_{\mathfrak{p}_i} (a + \sigma (a) )\geq  0$.
 
Finally, since $deg (\theta) = deg( \sigma (\theta ))$, we have $v_{\mathfrak{p}_\infty} ( a)=0$, so that $v_{\mathfrak{p}_\infty} (a + \sigma (a)) \geq 0$, concluding the proof of the lemma. 

\end{proof}

\bibliographystyle{abbrv}

\end{document}